\documentclass[11pt,reqno,oneside]{amsart}

\usepackage{amsmath, amsthm, amssymb}

\usepackage{tensor}
\usepackage{bbm}%lowercase mathbb
\usepackage{geometry}
\usepackage{soul}%underline 
\usepackage[pdftex]{graphicx}
\usepackage{float}
\usepackage{subfig}
\usepackage[english]{babel}
\usepackage{tikz} % graphical language
\usepackage{setspace}
\usetikzlibrary{arrows,decorations.pathreplacing,decorations.markings,shapes.geometric}
\tikzset{
    bt/.style={draw=blue,thick},
    ns/.style={circle,draw=blue,fill=blue, inner sep=0pt, minimum size=2mm},
    string/.style={draw=#1, postaction={decorate}, decoration={markings,mark=at position .45 with {\arrow[blue]{triangle 60}}}},
    doublestring/.style={draw=#1, postaction={decorate}, decoration={markings, mark=at position .7 with {\arrow[blue]{triangle 60}}, 
    mark=at position .3 with {\arrowreversed[blue]{triangle 60}}}},
    costring/.style={draw=#1, postaction={decorate}, decoration={markings,mark=at position .55 with {\arrow[draw=#1]{<}}}},
    arr/.style={string=blue, thick},
    doublearr/.style={doublestring=blue, thick},
    lin/.style={blue},
    dlin/.style = {blue, dashed, thick},
    dot/.style={circle,draw=#1,fill=#1,inner sep=1pt},
}

\usepackage{dsfont}
\usepackage{color}
\usepackage[color,matrix,arrow]{xy}%diagrams
\graphicspath{{graphics/}}
\numberwithin{equation}{section}
\usepackage{fancyhdr}
\theoremstyle{definition}

\newtheorem{Thm}{Theorem}[section]
\newtheorem{Cor}{Corollary}[section]
\newtheorem{Prop}{Proposition}[section]
\newtheorem{Lemma}{Lemma}[section]
\newtheorem{Rem}{Remark}[section]

\newtheorem{Not}{Notation}[section]

\newcommand{\sh}[1]{\mathcal{#1}}

\DeclareMathOperator{\Imag}{Im}

\DeclareMathOperator{\Uea}{\mathcal{U}}

\DeclareMathOperator{\SL}{SL}

\DeclareMathOperator{\C}{\mathbb{C}}

\DeclareMathOperator{\Z}{\mathbb{Z}}

\DeclareMathOperator{\Hom}{Hom}

\DeclareMathOperator{\codim}{codim}

\DeclareMathOperator{\Gm}{\mathbb{G}_m}
\DeclareMathOperator{\Ga}{\mathbb{G}_a}

\DeclareMathOperator{\Aff}{\mathbb{A}}

\DeclareMathOperator{\Lsimple}{L}
\DeclareMathOperator{\MVerma}{M}

\DeclareMathOperator{\RVerma}{R}

\usepackage[pdftex,
bookmarks=true,
bookmarksnumbered=true,
hypertexnames=false,
colorlinks = true,%colorlinks = true if you want color
linkcolor = black,
citecolor = black,
urlcolor = black]{hyperref}

\title{On a Gelfand-Tsetlin representation of $\mathfrak{sl}_3$ in the space of sections of
a  local system with two monodromy parameters}

\begin{document}

\author{Claude Eicher}
\address{SWITZERLAND}
\email{claudeeicher@gmail.com}
\date{\today}
\maketitle

\tableofcontents

\setstretch{1.4}

\begin{abstract}
We construct a Gelfand-Tsetlin representation of $\mathfrak{sl}_3$
in the space of sections of a local system. The local system lives on an open
part of the flag variety given by the intersection of three translates of the big cell
and has two complex monodromy parameters. We analyze the structure of
this representation. 
\end{abstract}

\section{Setup}
We consider the linear algebraic group $\SL_3$ over $\C$.
We fix a maximal torus $T$ and a Borel subgroup $B \supseteq T$ of it.
We denote by $U$ the unipotent radical of $B$. 
We have the Lie algebra $\mathfrak{sl}_3$ of $\SL_3$ and the triangular decomposition
$\mathfrak{sl}_3 = \mathfrak{n}^- \oplus \mathfrak{t}\oplus \mathfrak{n}
= \mathfrak{n}^- \oplus \mathfrak{b}$.
We denote by $\mathfrak{t}^* = \Hom_{\C}(\mathfrak{t},\C)$ the weights.
We fix a choice of root vectors 
$e_1, e_2, e_{12} = [e_1, e_2]$
associated to the roots $\alpha_1, \alpha_2, \alpha_1+\alpha_2$
of $\mathfrak{n}$
and $f_1, f_2, f_{12} = -[f_1, f_2]$ associated to the
roots $-\alpha_1,-\alpha_2, -\alpha_1-\alpha_2$ of $\mathfrak{n}^-$
such that the usual commutation relations are satisfied. 
The simple coroots are denoted by $h_1, h_2$. The Weyl vector 
is $\rho = \alpha_1+\alpha_2$. The simple reflections 
of the Weyl group are denoted by $s_1$ and $s_2$
and the longest element is $w_{\circ} = s_1 s_2 s_1$. 
We consider the flag variety $\SL_3/B$ of $\SL_3$. 
We have an isomorphism of varieties
\begin{align}\label{eq:Phi}
\begin{split}
& \Phi: U_{\alpha_1} \times U_{\alpha_2} \times U_{\alpha_1+\alpha_2}
\xrightarrow{\cong} X_{w_{\circ}} =U w_{\circ}B\\
& (e^{x_{\alpha_1}e_1}, e^{x_{\alpha_2}e_2},
e^{x_{\alpha_1+\alpha_2}e_{12}}) \mapsto e^{x_{\alpha_1}e_1}e^{x_{\alpha_2}e_2}e^{x_{\alpha_1+\alpha_2}e_{12}} w_{\circ} B
\end{split}
\end{align}
parametrizing the big cell $X_{w_{\circ}}$ in $\SL_3/B$, which is the open $B$-orbit in that variety. 
Here $U \supseteq U_{\alpha} \cong \Ga$ denotes the root subgroup associated to the root $\alpha$. 

\section{Local system $\sh{L}_{\mu_1, \mu_2}$ and its space of sections $M_{\mu_1, \mu_2}$}
We define $O = X_{w_{\circ}} \cap s_1 X_{w_{\circ}} \cap s_2 X_{w_{\circ}}$. 
\eqref{eq:Phi} restricts to an isomorphism 
\begin{align*}
\Phi: \Gm^2 \times \Aff^1 = (U_{\alpha_1}\setminus 1)\times
(U_{\alpha_2} \setminus 1) \times U_{\alpha_1+\alpha_2} \xrightarrow{\cong} O
\end{align*}
of varieties. 
We define for $\mu_1, \mu_2 \in \C$ the local system $\sh{L}_{\mu_1, \mu_2}$
in the category of (left) $\sh{D}$-modules on $O$ by
$\Phi^*(\sh{L}_{\mu_1, \mu_2}) = \sh{O}^{(\mu_1)}_{\Gm} \boxtimes \sh{O}^{(\mu_2)}_{\Gm}
\boxtimes \sh{O}_{\Aff^1}$. Here we use the definition $\sh{O}^{(\mu)}_{\Gm} =
\sh{D}_{\Gm}/\sh{D}_{\Gm}(x\partial_x+\mu)$, where $\sh{D}_{\Gm}$ is the sheaf 
of differential operators on $\Gm$. 
We refer to $\mu_1, \mu_2$ as the two monodromy parameters of $\sh{L}_{\mu_1, \mu_2}$.
By definition $\sh{L}_{\mu_1, \mu_2}$ is isomorphic to $\sh{L}_{\mu_1^{\prime}, \mu_2^{\prime}}$
if and only if $\mu_1-\mu_1^{\prime} \in \Z$ and $\mu_2 -\mu_2^{\prime} \in \Z$. 
The image of $1 \in \sh{D}_{\Gm}$
is denoted by $1_{\mu} \in \sh{O}^{(\mu)}_{\Gm}$. Then $x^k 1_{\mu}$, $k \in \Z$,
is a $\C$-basis of $\Gamma(\Gm, \sh{O}^{(\mu)}_{\Gm})$ and we have
\begin{align}\label{eq:partialxonxk1mu}
\partial_x x^k 1_{\mu}  = (k-\mu) x^{k-1}1_{\mu}\;.
\end{align}
The space of sections of $\sh{L}_{\mu_1, \mu_2}$ will be denoted by $M_{\mu_1, \mu_2}
= \Gamma(O,\sh{L}_{\mu_1, \mu_2})$. $M_{\mu_1, \mu_2}$
has the $\C$-basis 
\begin{align*}
u_{k,l,m} = x_{\alpha_1}^k 1_{\mu_1} \otimes x_{\alpha_2}^l 1_{\mu_2}
\otimes x^m_{\alpha_1+\alpha_2},\; k,l \in \Z,\; m \in \Z_{\geq 0}\;.
\end{align*}
$M_{\mu_1, \mu_2}$ is acted upon by the Lie algebra $\mathfrak{sl}_3$ through infinitesimal left translations
given by the Lie algebra homomorphism $a: \mathfrak{sl}_3 \rightarrow V$,
where $V$ denotes the vector fields on $O$. The $\mathfrak{sl}_3$-module $M_{\mu_1, \mu_2}$ has central character zero. 
By the Beilinson-Bernstein equivalence \cite{BB81} $M_{\mu_1, \mu_2}$
is isomorphic to $M_{\mu_1^{\prime}, \mu_2^{\prime}}$ as $\mathfrak{sl}_3$-module if and only if $\mu_1-\mu_1^{\prime} \in \Z$
and $\mu_2-\mu_2^{\prime} \in \Z$. Thus, the isomorphism class of
the $\mathfrak{sl}_3$-module $M_{\mu_1, \mu_2}$
only depends on $\mu_1+\Z$ and $\mu_2+\Z$.  

\section{$\mathfrak{sl}_3$-action on the basis $u_{k,l,m}$}
We compute that
\begin{align*}
a(e_1) &= -\partial_{x_{\alpha_1}} \\
a(e_2) &= -\partial_{x_{\alpha_2}}+x_{\alpha_1}\partial_{x_{\alpha_1+\alpha_2}} \\
a(h) &= -\alpha_1(h) x_{\alpha_1} \partial_{x_{\alpha_1}}-\alpha_2(h)x_{\alpha_2}\partial_{x_{\alpha_2}}-(\alpha_1+\alpha_2)(h)x_{\alpha_1+\alpha_2}\partial_{x_{\alpha_1+\alpha_2}},\; h \in \mathfrak{t} \\
a(f_1) &= x_{\alpha_1}^2\partial_{x_{\alpha_1}}-(x_{\alpha_1+\alpha_2}+x_{\alpha_1}x_{\alpha_2})\partial_{x_{\alpha_2}}
+x_{\alpha_1}x_{\alpha_1+\alpha_2}\partial_{x_{\alpha_1+\alpha_2}} \\
a(f_2) &= x_{\alpha_2}^2\partial_{x_{\alpha_2}}+x_{\alpha_1+\alpha_2}\partial_{x_{\alpha_1}}\\
a(e_{12}) &= -\partial_{x_{\alpha_1+\alpha_2}} \\
a(f_{12}) &= x_{\alpha_1}x_{\alpha_1+\alpha_2}\partial_{x_{\alpha_1}}+x_{\alpha_2}(x_{\alpha_1+\alpha_2}+x_{\alpha_1}x_{\alpha_2})\partial_{x_{\alpha_2}}+x_{\alpha_1+\alpha_2}^2\partial_{x_{\alpha_1+\alpha_2}}\;.
\end{align*}
From this and \eqref{eq:partialxonxk1mu} we compute

\begin{align}\label{eq:actionofgeneratorsonbasis}
\begin{split}
e_1 u_{k,l,m} &= -\overline{k} u_{k-1,l,m} \\
e_2 u_{k,l,m} &= -\overline{l}u_{k,l-1,m} +m u_{k+1,l,m-1} \\
h u_{k,l,m} &= -((\overline{k}+m)\alpha_1(h)+(\overline{l}+m)\alpha_2(h))u_{k,l,m},\; h \in \mathfrak{t} \\
f_1 u_{k,l,m} &= (\overline{k}-\overline{l}+m) u_{k+1,l,m} -\overline{l}u_{k,l-1,m+1} \\
f_2 u_{k,l,m} &= \overline{l}u_{k,l+1,m}+\overline{k}u_{k-1,l,m+1} \\
e_{12} u_{k,l,m} &= -m u_{k,l,m-1} \\
f_{12} u_{k,l,m} &= (\overline{k}+\overline{l}+m) u_{k,l,m+1}
+\overline{l}u_{k+1, l+1, m}\;.
\end{split}
\end{align}
\emph{Here and below we abbreviate $\overline{k} = k-\mu_1$ and $\overline{l} = l-\mu_2$.}
In particular we see that the subspace of $M_{\mu_1, \mu_2}$ annihilated by $e_{12}$ has the basis $u_{k,l,0}, k, l \in \Z$, and that the $\mathfrak{t}$-weight of $u_{k,l,m}$ is 
\begin{align} \label{eq:weightofuklm}
-(\overline{k}+m)\alpha_1-(\overline{l}+m)\alpha_2\;.
\end{align}
About the structure of the linear combinations in \eqref{eq:actionofgeneratorsonbasis}  that involve two
terms we observe the following: The shift of the indices $k,l,m$ in the two terms is always in complementary components. 

\section{Basis $w_{k, l, m}$}

\emph{For the rest of the article we assume $\mu_1+\mu_2 \notin \Z$ as most of the results require
this condition.} We introduce for $k,l \in \Z, m \in \Z_{\geq 0}$ the vectors in $M_{\mu_1, \mu_2}$
\begin{align} \label{eq:wklm}
w_{k, l, m} = \sum_{n = 0}^m \binom{m}{n}\frac{\overline{l}^{(n)}}{(\overline{k}+\overline{l})^{(n)}} u_{k+n, l+n, m-n} = u_{k,l, m}+\dots\;.
\end{align}
Here $x^{(n)} = x(x+1)\cdots (x+n-1)$ denotes a raising factorial. We have
$w_{k,l,0} = u_{k,l,0}$. We have $f_{12}e_{12} w_{k,l, m} = -m(\overline{k}+\overline{l}+m-1) w_{k,l,m}$
and this eigenvalue condition suffices to determine the coefficients in \eqref{eq:wklm} 
using \eqref{eq:actionofgeneratorsonbasis}. We also have
$h_1 w_{k,l,m} = (-2\overline{k}+\overline{l}-m)w_{k,l,m}$
and $h_2 w_{k,l,m}=(\overline{k}-2\overline{l}-m)w_{k,l,m}$ and the $\mathfrak{t}$-weight
of $w_{k,l,m}$ is \eqref{eq:weightofuklm}. 
We also note that if $-\overline{l} \in \{0, 1, \dots, m\}$ the sum in \eqref{eq:wklm}
truncates to $\sum_{n=0}^{\min\{m,-\overline{l}\}}$. 

\begin{Lemma}
$w_{k, l, m}$, $k, l \in \Z, m \in \Z_{\geq 0}$, form a basis of $M_{\mu_1, \mu_2}$
of simultaneous eigenvectors with respect to the triple of endomorphisms $(h_1, h_2, f_{12}e_{12})$.
The simultaneous eigenvalues determine $(k,l,m)$ uniquely. 
\end{Lemma}

\begin{proof}
We first show that the simultaneous eigenvalues
$(-2\overline{k}+\overline{l}-m, \overline{k}-2\overline{l}-m,-m(\overline{k}+\overline{l}+m-1))$
of $w_{k,l,m}$ with respect to $(h_1,h_2, f_{12}e_{12})$ determine $(k,l,m)$ uniquely. 
Hence we assume
\begin{align*}
-2\overline{k}+\overline{l}-m &= -2\overline{k^{\prime}}+\overline{l^{\prime}}-m^{\prime} \\
-2\overline{l}+\overline{k}-m &= -2\overline{l^{\prime}}+\overline{k^{\prime}}-m^{\prime} \\
-m(m+\overline{k}+\overline{l}-1) &= -m^{\prime}(m^{\prime}+\overline{k^{\prime}}+\overline{l^{\prime}}-1)\;.
\end{align*}
From the first two equations we get $\overline{k} = \overline{k^{\prime}}+m^{\prime}-m$
and $\overline{l} =\overline{l^{\prime}}+m^{\prime}-m$. If we plug this into the last equation
we obtain the equation
\begin{align*}
0 = (m-(\overline{k^{\prime}}+\overline{l^{\prime}}+m^{\prime}-1))(m-m^{\prime})
\end{align*}
for the remaing variable $m$. We find the solutions
\begin{align*}
\begin{split}
& m = m^{\prime} \quad \overline{k}= \overline{k^{\prime}} \quad \overline{l} =\overline{l^{\prime}}\\
& m = \overline{k^{\prime}}+\overline{l^{\prime}}+m^{\prime}-1\quad \overline{k} = 1-\overline{l^{\prime}} \quad
\overline{l} = 1-\overline{k^{\prime}}\;.
\end{split}
\end{align*}
The second solution, however, requires $\mu_1 +\mu_2 \in \Z$ 
and hence only the first is left.  It follows that the $w_{k,l,m}$ are linearly independent
since they have distinct simultaneous eigenvalues. On the other hand, by induction
on $m \in \Z_{\geq 0}$ it is clear from \eqref{eq:wklm} that $u_{k,l,m}$ is a linear combination of the 
$w_{k+n,l+n,m-n}$ for $n \in \{0, \dots, m\}$. Thus the 
$w_{k,l,m}$ span $M_{\mu_1, \mu_2}$. We have
shown that the $w_{k,l, m}$ form a basis of $M_{\mu_1, \mu_2}$. 
\end{proof}

The lemma implies that $M_{\mu_1, \mu_2}$ is a Gelfand-Tsetlin module
and the $w_{k, l, m}$ a Gelfand-Tsetlin basis. Namely,
it is a direct sum of one dimensional simultaneous eigenspaces $\C w_{k, l, m}$ with respect to $\Gamma$,
where $\Gamma$ is the Gelfand-Tsetlin commutative subalgebra of $\Uea \mathfrak{sl}_3$ determined
by $h_1, h_2$, and the Casimir operator of the $\mathfrak{sl}_2$-subalgebra associated
to the root $\alpha_1+\alpha_2$. Here $\Uea \mathfrak{sl}_3$ is the universal enveloping
algebra of $\mathfrak{sl}_3$.
 Thus $\Gamma$ is a polynomial algebra in five generators
and its action on $M_{\mu_1, \mu_2}$ factors through the polynomial subalgebra whose
generators can be taken to be $h_1, h_2, f_{12}e_{12}$,
the triple we have considered above. We will denote it by $\Gamma_0 = \C[h_1, h_2, f_{12}e_{12}]$.

For $k, l \in \Z, m \in \Z_{\geq 0}$ we can invert \eqref{eq:wklm} to
\begin{align*}
u_{k, l, m} = \sum_{n = 0}^{m}(-1)^n \binom{m}{n} \frac{\overline{l}^{(n)}}{(\overline{k}+\overline{l}+n-1)^{(n)}} w_{k+n, l+n, 
m-n} = w_{k, l, m}+\dots\;.
\end{align*}

\section{$\mathfrak{sl}_3$-action on $w_{k,l,m}$}

We deduce from \eqref{eq:actionofgeneratorsonbasis} that
\begin{align}\label{eq:actionofbasisofsl3onwklm}
\begin{split}
e_1 w_{k,l,m} &= -\overline{k} w_{k-1,l, m}-m\frac{\overline{l}(\overline{l}-1)}{(\overline{k}+\overline{l})
(\overline{k}+\overline{l}-1)}w_{k,l+1,m-1} \\
e_2 w_{k,l,m}&= -\overline{l} w_{k,l-1,m}+m\frac{\overline{k}(\overline{k}-1)}{(\overline{k}+\overline{l})
(\overline{k}+\overline{l}-1)}w_{k+1,l,m-1} \\
f_1 w_{k,l,m}&= \frac{\overline{k}(\overline{k}-1)(\overline{k}+\overline{l}+m)}{(\overline{k}+\overline{l})(\overline{k}+\overline{l}-1)}w_{k+1,l,m}-\overline{l} w_{k,l-1,m+1} \\
f_2 w_{k,l,m} &= \frac{\overline{l}(\overline{l}-1)(\overline{k}+\overline{l}+m)}{(\overline{k}+\overline{l})(\overline{k}+\overline{l}-1)}w_{k,l+1,m} +\overline{k} w_{k-1,l,m+1} \\
e_{12}w_{k,l,m} &= -m w_{k,l,m-1} \\
f_{12} w_{k,l,m}&= (\overline{k}+\overline{l}+m) w_{k,l,m+1}\;.
\end{split}
\end{align}
Here are some basic remarks about the structure of these formulae.
If the right-hand side involves two terms, the shift of $k,l,m$ in these two terms is in 
complementary components. 
The formulae are symmetric under the simultaneous interchange $1 \leftrightarrow 2$ and $k \leftrightarrow l$, if we switch the sign in the term where $m$ is shifted. 
For $m=0$ the action of $e_1$ and $e_2$ involves only one basis element
and one of the coefficients in the expression $f_1 w_{k,l,m}$ and $f_2 w_{k,l,m}$ simplifies. 
 
 \section{Simplicity of $M_{\mu_1, \mu_2}$}

We recall the well-known and elementary 
\begin{Rem} \label{Rem:simultaneousevecslieinsubmodule}
Let $M$ be an $\mathfrak{sl}_3$-module. We consider $v = \sum_{j=1}^n v_j$, where $v_j \in M$ are simultaneous $(h_1,h_2, f_{12}e_{12})$-eigenvectors,
in particular nonzero, with distinct simultaneous eigenvalues. 
Let $\chi_j: \Gamma_0 \rightarrow \C$ be the character defined by the simultaneous eigenvalue of
$v_j$. Let $S$ be the $\Gamma_0$-submodule of $M$ generated by $v$. 
Then each $v_j \in S$. In fact, $S$ equals the $\Gamma_0$-submodule generated by $v_1, \dots, v_n$, in particular it is generated by simultaneous eigenvectors. If $S$ is instead the $\mathfrak{sl}_3$-submodule of $M$ generated by $v$, then $S$ equals the $\mathfrak{sl}_3$-submodule generated by
$v_1, \dots, v_n$. 
Indeed, we have for $g \in \Gamma_0$
\begin{align*}
S \ni v &= \sum_j v_j\\
S \ni g v &= \sum_j \chi_j(g)v_j \\
\cdots \\
S \ni g^n v &= \sum_j \chi_j(g)^n v_j\;.
\end{align*}
If the Vandermonde determinant 
\begin{align}\label{eq:VanderMonde}
\prod_{i < j} (\chi_j(g)-\chi_i(g)) \neq 0
\end{align}
it follows that each $v_j \in S$. 
Now for $i \neq j$ distinctness of the simultaneous eigenvalues of $v_i$ and $v_j$ is equivalent to $\chi_j \neq \chi_i$
and to the linear functional 
\begin{align*}
\chi_{ij} =(\chi_j-\chi_i)\vert \C h_1 \oplus \C h_2 \oplus \C f_{12}e_{12}: \C h_1 \oplus \C h_2 \oplus \C f_{12}e_{12}  \rightarrow \C
\end{align*}
being nonzero. Hence $\codim \ker \chi_{ij} \geq 1$ and consequently $\C h_1 \oplus \C h_2 \oplus \C f_{12}e_{12} \setminus \bigcup_{i < j} \ker \chi_{ij}$
is open in $\C h_1 \oplus \C h_2 \oplus \C f_{12}e_{12}$. Thus, there is a $g \in \C h_1 \oplus \C h_2 \oplus \C f_{12}e_{12}$ such that
\eqref{eq:VanderMonde} holds. 
\end{Rem}

\begin{Thm}\label{Thm:Mmu1mu2simplicitycriterion}
$M_{\mu_1, \mu_2}$ is simple if and only if $\mu_1 \notin \Z$ and $\mu_2 \notin \Z$. 
\end{Thm}

\begin{proof}
If $\mu_1 \in \Z$ or $\mu_2 \in \Z$ it is easy to exhibit proper
submodules of $M_{\mu_1, \mu_2}$, see \autoref{sec:subquotients} below. Now let $\mu_1 \notin \Z$ and $\mu_2 \notin \Z$. Let $v \neq 0 \in M_{\mu_1, \mu_2}$ and $S$ be the $\mathfrak{sl}_3$-submodule generated by $v$. 
By Remark \autoref{Rem:simultaneousevecslieinsubmodule} $w_{k_0, l_0, m_0} \in S$
for some $k_0, l_0, m_0$. 
Hence it suffices to show that $w_{k_0, l_0, m_0}$ generates all of $M_{\mu_1, \mu_2}$
when acting with $\mathfrak{sl}_3$. 
We now look at \eqref{eq:actionofbasisofsl3onwklm}. Acting with $e_{12}$ we find $w_{k_0, l_0, 0} \in S$. 
Acting then with $e_1$ we find $w_{k, l_0, 0} \in S$ for $k \leq k_0$. Acting with $e_2$ we find 
$w_{k, l, 0} \in S$ for $k \leq k_0, l \leq l_0$. Acting with $f_{12}$ we find $w_{k,l,m} \in S$, $k \leq k_0, l \leq l_0, m \geq 0$. 
Now $f_1 w_{k_0, l, m} = C w_{k_0+1, l, m} - \overline{l}w_{k_0, l-1, m+1}$ 
with $C \neq 0$ and this implies $w_{k_0+1, l, m} \in S$ for $l \leq l_0, m \geq 0$. 
Thus we find $w_{k,l,m} \in S$ for $k \in \Z, l \leq l_0, m \geq 0$. 
Using $f_2$ we find in the same way $w_{k,l,m}\in S$ for $k, l \in \Z, m \geq 0$. 
\end{proof}

Let $j: O \hookrightarrow \SL_3/B$ be the inclusion. As a direct consequence
of the Beilinson-Bernstein equivalence \cite{BB81} we find
\begin{Cor}
$j_{\cdot} \sh{L}_{\mu_1, \mu_2}$ is a simple $\sh{D}_{\SL_3/B}$-module
if and only if $\mu_1 \notin \Z$ and $\mu_2 \notin \Z$. Thus, the extension of $\sh{L}_{\mu_1, \mu_2}$ with respect to $j$ is clean \cite{Ber}
if and only if $\mu_1 \notin \Z$ and $\mu_2 \notin \Z$. 
\end{Cor}

\section{Dual $M_{\mu_1, \mu_2}^{\vee}$}

\subsection{Definition of the dual}
Let $M$ be an $\mathfrak{sl}_3$-module with the property that it is the direct sum
of its finite dimensional simultaneous $(h_1, h_2, f_{12}e_{12})$-eigenspaces  $M_{\chi}$.
We define on $M^{\vee} = \bigoplus_{\chi} \Hom_{\C}(M_{\chi},\C)$, where $\chi$
runs through the simultaneous eigenvalues of $M$, an $\mathfrak{sl}_3$-action by
$(X \eta)(v) = -\eta(\tau(X)v)$ for $X \in \mathfrak{sl}_3$, $v \in M_{\chi}$, $\eta
\in \Hom_{\C}(M_{\chi}, \C)$. Here $\tau$ is the standard involution of $\mathfrak{sl}_3$
given by 
\begin{align*}
\tau: h_1 \mapsto -h_1,\; h_2 \mapsto -h_2,\; f_1 \leftrightarrow e_1, f_2 \leftrightarrow e_2,\;
-f_{12} \leftrightarrow e_{12}\;.
\end{align*}

The vector subspace $\Hom_{\C}(M_{\chi},\C) \subseteq M^{\vee}$ is then the simultaneous
eigenspace associated to $\chi$, i.e. $\Hom_{\C}(M_{\chi}, \C) = (M^{\vee})_{\chi}$. 
$M^{\vee}$ is called the dual of $M$.

\subsection{$\mathfrak{sl}_3$-action on the dual basis $\eta_{k,l,m}$}

This construction of the dual applies to $M_{\mu_1, \mu_2}$. 
We consider the basis $\eta_{k,l,m}$ of $M_{\mu_1,\mu_2}^{\vee}$ dual to the basis $w_{k,l,m}$ of $M_{\mu_1,\mu_2}$. Thus, $\eta_{k,l,m} \in \Hom_{\C}(\C w_{k,l,m},\C)$
is defined by the normalization $\eta_{k,l,m}(w_{k,l,m}) = 1$. As observed above, $\eta_{k,l,m}$ has the same simultaneous
$(h_1,h_2,f_{12}e_{12})$-eigenvalues as $w_{k,l,m}$. 
The $\mathfrak{sl}_3$-action 
on the $\eta_{k,l,m}$ immediately follows from \eqref{eq:actionofbasisofsl3onwklm}
\begin{align}\label{eq:actiononetaabc}
\begin{split}
e_1 \eta_{k,l,m} &= -\frac{(\overline{k}-1)(\overline{k}-2)(\overline{k}+\overline{l}+m-1)}{(\overline{k}+\overline{l}-1)(\overline{k}+\overline{l}-2)}\eta_{k-1,l,m} +\begin{cases}(\overline{l}+1)\eta_{k,l+1,m-1} & m > 0 \\ 0 & m=0 \end{cases}\\
e_2 \eta_{k,l,m}&= -\frac{(\overline{l}-1)(\overline{l}-2)(\overline{k}+\overline{l}+m-1)}{(\overline{k}+\overline{l}-1)(\overline{k}+\overline{l}-2)}\eta_{k,l-1,m}
-\begin{cases}(\overline{k}+1)\eta_{k+1,l,m-1} & m > 0 \\ 0& m=0\end{cases} \\
f_1 \eta_{k,l,m} &= (\overline{k}+1)\eta_{k+1,l,m}+(m+1)\frac{(\overline{l}-1)(\overline{l}-2)}{(\overline{k}+\overline{l}-1)(\overline{k}+\overline{l}-2)}\eta_{k,l-1,m+1}\\
f_2 \eta_{k,l,m} &= (\overline{l}+1)\eta_{k,l+1,m}-(m+1)\frac{(\overline{k}-1)(\overline{k}-2)}{(\overline{k}+\overline{l}-1)(\overline{k}+\overline{l}-2)}\eta_{k-1,l,m+1} \\
e_{12} \eta_{k,l,m} &= \begin{cases}(\overline{k}+\overline{l}+m-1)\eta_{k,l,m-1} & m > 0 \\0 & m = 0 \end{cases}\\
f_{12}\eta_{k,l,m}&= -(m+1)\eta_{k,l,m+1}\;.
\end{split}
\end{align}

\subsection{Structure of $M^{\vee}_{\mu_1, \mu_2}$}

It follows from Theorem \autoref{Thm:Mmu1mu2simplicitycriterion} that $M_{\mu_1, \mu_2}^{\vee}$ is simple if and only if
$\mu_1 \notin \Z$ and $\mu_2 \notin \Z$. 

\begin{Thm}
If $\mu_2 \in \Z$ and $k_0 \in \Z$ $\eta_{k_0, \mu_2, 0}$ generates $M^{\vee}_{\mu_1, \mu_2}$
as $\mathfrak{sl}_3$-module. The analogous statement holds for $\mu_1 \in \Z$. In particular, $M^{\vee}_{\mu_1, \mu_2}$
is a cyclic $\mathfrak{sl}_3$-module for any $\mu_1, \mu_2$. 
\end{Thm}

\begin{proof}
This follows by inspecting \eqref{eq:actiononetaabc}. Let $S$ be the $\mathfrak{sl}_3$-submodule
generated by $\eta_{k_0, \mu_2, 0}$. Acting with $e_2$ we get $\eta_{k_0, l, 0} \in S$, $\overline{l} \leq 0$. Acting then with $e_1$ we get $\eta_{k,l,0} \in S$, $k \leq k_0, \overline{l} \leq 0$. 
Acting then with $f_{12}$ we get $\eta_{k,l,m} \in S$, $k \leq k_0, \overline{l} \leq 0, m \geq 0$. 
From $f_1 \eta_{k_0, l, m} = (\overline{k_0}+1)\eta_{k_0+1, l, m}+ C \eta_{k_0, l-1,m+1}$
and $f_1 \eta_{k_0, l, m}, \eta_{k_0,l-1,m+1} \in S$ we conclude $\eta_{k_0+1,l,m} \in S$
for $\overline{l} \leq 0, m \geq 0$. Thus $\eta_{k,l,m}\in S$, $k \in \Z, \overline{l} \leq 0, m \geq 0$. 
Similarly, $f_2 \eta_{k,l,m} = (\overline{l}+1)\eta_{k,l+1,m}+D \eta_{k-1,l,m+1}$ shows,
as $\overline{l}+1 > 0$ for $\overline{l} \geq 0$, that $\eta_{k,l,m} \in S$,
$k,l \in \Z, m \geq 0$. \par
In the case $\mu_1 \in \Z$ we argue as in the above argument, applying
$e_1, e_2, f_{12}, f_2, f_1$, in this order, in accord with the symmetry of the formulae \eqref{eq:actiononetaabc}. 
\end{proof}

 \begin{Thm} \label{Thm:M-mu1mu2congMveemu1mu2}
$M_{\mu_1,\mu_2} \cong M_{\mu_1,\mu_2}^{\vee}$ as $\mathfrak{sl}_3$-module if and only if $\mu_1 \notin \Z$ and $\mu_2 \notin \Z$.
 \end{Thm}
 
 \begin{proof}
If $\mu_1 \notin \Z$ and $\mu_2 \notin \Z$ we have for $x_{k,l,m}\in \C^{\times}$ given by
\begin{align}\label{eq:xabc}
x_{k,l,m} = (-1)^m\frac{(\overline{k}+\overline{l})^{(m)} \mu_1 \mu_2}{\overline{k}\; \overline{l}\; (m!)}
\times \begin{cases}\frac{(-\mu_1-1)^{(k)}}{(-\mu_1-\mu_2-1)^{(k)}}  & k \geq 0 \\ 
\frac{(-2-\mu_1-\mu_2)_{(-k)}}{(-2-\mu_1)_{(-k)}} & k < 0\end{cases} 
\times \begin{cases} \frac{(-\mu_2-1)^{(l)}}{(\overline{k}-\mu_2-1)^{(l)}} &  l \geq 0 \\
\frac{(\overline{k}-\mu_2-2)_{(-l)}}{(-2-\mu_2)_{(-l)}} & l < 0 \end{cases}
\end{align}
(raising respectively falling factorials in the first respectively second cases) that,
if we set $x_{k,l,m} \eta_{k,l,m}^{\prime}= \eta_{k,l,m}$,
\eqref{eq:actionofbasisofsl3onwklm} is satisfied with $\eta_{k,l,m}^{\prime}$ replacing $w_{k,l,m}$. 
 In fact, the $x_{k,l,m}$ are uniquely determined by the initial value $x_{0,0,0} = 1$ together with the 
 requirement that 
 \begin{align*}
 e_1 \eta^{\prime}_{k,l,m}= -\overline{k} \eta^{\prime}_{k-1,l,m}+ \dots\quad 
 e_2 \eta^{\prime}_{k,l,m} = -\overline{l} \eta^{\prime}_{k,l-1,m}+\dots\quad e_{12} \eta^{\prime}_{k,l,m}= -m \eta^{\prime}_{k,l,m-1}\;,
 \end{align*}
 as this fixes 
 $\frac{x_{k,l,m}}{x_{k-1,l,m}} = \frac{(\overline{k}-1)(\overline{k}-2)(\overline{k}+\overline{l}+m-1)}{\overline{k}(\overline{k}+\overline{l}-1)(\overline{k}+\overline{l}-2)}$, 
 $\frac{x_{k,l,m}}{x_{k,l-1,m}} = \frac{(\overline{l}-1)(\overline{l}-2)(\overline{k}+\overline{l}+m-1)}{\overline{l}(\overline{k}+\overline{l}-1)(\overline{k}+\overline{l}-2)}$, $\frac{x_{k,l,m}}{x_{k,l,m-1}} = -\frac{\overline{k}+\overline{l}+m-1}{m}$
(of course we need $m>0$ for the last quotient), respectively. This shows the if-part of the statement. \par
We now prove the only if-part.  We assume $\mu_1 \in \Z$. 
 It suffices to compare the expression for $e_1 \eta_{k,l,m}$
with the one for $e_1 w_{k,l,m}$. There is a $k \in \Z$ such that $\overline{k} = 0$ 
and the coefficient of $w_{k-1,l,m}$ vanishes, while
the one of $\eta_{k-1,l,m}$ does not. 
Any isomorphism of $\mathfrak{sl}_3$-modules $\phi: M_{\mu_1,\mu_2} \xrightarrow{\cong} M_{\mu_1,\mu_2}^{\vee}$ satisfies
$\phi(w_{k,l,m})= x_{k,l,m} \eta_{k,l,m}$ for some $x_{k,l,m} \in \C^{\times}$. But the
above shows that
there is no such $x_{k,l,m}$ for $\overline{k}=0$. In
the case $\mu_2 \in \Z$ one argues similarly. 
 \end{proof}
 
 \section{Subquotients of $M_{\mu_1, \mu_2}$} \label{sec:subquotients}
 
 \emph{For the rest of the article we assume $\mu_2 \in \Z$}. By our standing assumption
$\mu_1+\mu_2 \notin \Z$ it follows $\mu_1 \notin \Z$. The statements in the 
complementary case $\mu_1 \in \Z$ and $\mu_2 \notin \Z$
are of course analogous. 
 
 \begin{Not} \label{Not:Mmu1mu2S}
Let $J \subseteq \Z^2 \times \Z_{\geq 0}$ be a subset. If $M_{\mu_1, \mu_2, J} = \bigoplus_{(k,l,m) \in J}\C w_{k,l,m}$ is a $\mathfrak{sl}_3$-submodule of $M_{\mu_1, \mu_2}$, then the quotient module can be identified with $M_{\mu_1, \mu_2, (\Z^2 \times \Z_{\geq 0})\setminus J} = \bigoplus_{(k,l,m) \in  (\Z^2 \times \Z_{\geq 0})\setminus J}\C w_{k,l,m}$. More generally, any subquotient of the $\mathfrak{sl}_3$-module $M_{\mu_1, \mu_2}$ with underlying vector space $\bigoplus_{(k,l,m) \in J}\C w_{k,l,m}$
is denoted by $M_{\mu_1, \mu_2, J}$.
 In particular, the notation $M_{\mu_1, \mu_2, J}$ does not specify whether the module is a submodule, quotient,
or a more general subquotient of $M_{\mu_1, \mu_2}$. 
\end{Not}

As a consequence of Remark \autoref{Rem:simultaneousevecslieinsubmodule} any $\mathfrak{sl}_3$-submodule of $M_{\mu_1, \mu_2}$ is of the form $M_{\mu_1, \mu_2, J}$ for some $J$. 
We note that the $\mathfrak{sl}_3$-action on $M_{\mu_1, \mu_2, J}$ is obtained by truncating
the formulae \eqref{eq:actionofbasisofsl3onwklm}, i.e. by omitting the terms on the right-hand side if
the indices do not belong to $J$. We also note that the same construction applies to
$M^{\vee}_{\mu_1, \mu_2}$ instead of $M_{\mu_1, \mu_2}$, i.e. we may form $(M^{\vee}_{\mu_1, \mu_2})_J$,
and we may identify $(M^{\vee}_{\mu_1, \mu_2})_J = (M_{\mu_1, \mu_2, J})^{\vee}$.  
Thus we simply write $M^{\vee}_{\mu_1, \mu_2, J}$. 

\begin{Rem} \label{Rem:isomclassofsubquotients}
Let $\mu_2^{\prime} \in \Z$ and $\mu_1^{\prime}$ such that $\mu_1^{\prime}-\mu_1 \in \Z$. 
An isomorphism $M_{\mu_1, \mu_2} \cong M_{\mu_1^{\prime}, \mu_2^{\prime}}$ 
of $\mathfrak{sl}_3$-modules
induces an isomorphism 
of the subquotients $M_{\mu_1, \mu_2, J} \cong M_{\mu_1^{\prime}, \mu_2^{\prime},
J+(\mu_1^{\prime}-\mu_1, \mu_2^{\prime}-\mu_2, 0)}$.
\end{Rem}

\emph{We will often abbreviate $M = M_{\mu_1, \mu_2}$ from now on.}
 
 \begin{Rem} \label{Rem:subquotientsofMmu1mu2}
 We inspect \eqref{eq:actionofbasisofsl3onwklm} and conclude that
 $M$ has the following subquotients 
 
\begin{enumerate}
\item submodules $M_{\overline{l} \geq 0}$,
$M_{\overline{l}=0}$, $M_{\overline{l} \in \{0, 1\}}$,
$M_{\overline{l} \leq 1}$, $M_{\overline{l} \leq 0}$
\item quotients $M_{\overline{l} \geq 1}$, $M_{\overline{l} \geq 2}$,
 $M_{\overline{l} \leq -1}$
\item subquotient $M_{\overline{l}=1}$
that is not a submodule or quotient.
\end{enumerate}

Here we make use of Notation \autoref{Not:Mmu1mu2S},
for instance $\overline{l} \geq 0$ abbreviates the subset
\begin{align*}
J = \{(k,l,m) \in \Z^2 \times \Z_{\geq 0}\; \vert\; l \geq \mu_2\}\;.
\end{align*}
Of course, the above subquotients are related
to each other in various ways, for instance $M_{\overline{l}=1}$ is 
a submodule of $M_{\overline{l} \geq 1}$. 
Remark \autoref{Rem:isomclassofsubquotients} shows that the isomorphism
class of each these subquotients only depends on $\mu_1+\Z$ and is independent of $\mu_2$. 
\end{Rem}

\subsection{$\mathfrak{sl}_3$-action in $M_{ \overline{l} \in \{0, 1\}}$}
In $M_{\overline{l} \in \{0, 1\}}$ \eqref{eq:actionofbasisofsl3onwklm} reduces to
\begin{align} \label{eq:actiononwmu2andmu2+1}
\begin{split}
e_1 w_{k, \mu_2, m}&= -\overline{k}w_{k-1, \mu_2, m} \\
e_1 w_{k, \mu_2+1,m}&= -\overline{k} w_{k-1, \mu_2+1,m} \\
e_2 w_{k, \mu_2, m}&= m w_{k+1, \mu_2, m-1} \\
e_2 w_{k, \mu_2+1,m} &= -w_{k, \mu_2, m}+ m \frac{\overline{k}-1}{\overline{k}+1} w_{k+1, \mu_2+1, m-1}\\
f_1 w_{k, \mu_2, m} &= (\overline{k}+m) w_{k+1, \mu_2, m}\\
f_1 w_{k, \mu_2+1, m} &= \frac{(\overline{k}-1)(\overline{k}+m+1)}{\overline{k}+1} w_{k+1, \mu_2+1, m}-w_{k, \mu_2, m+1}\\
f_2 w_{k, \mu_2, m}&= \overline{k} w_{k-1, \mu_2, m+1} \\
f_2 w_{k, \mu_2+1, m}&= \overline{k}w_{k-1, \mu_2+1, m+1}\\
e_{12} w_{k, \mu_2, m} &= -m w_{k, \mu_2, m-1} \\
e_{12} w_{k, \mu_2+1, m}&= -m w_{k, \mu_2+1, m-1}\\
f_{12} w_{k, \mu_2, m} &= (\overline{k}+m) w_{k, \mu_2, m+1} \\
f_{12} w_{k, \mu_2+1, m} &= (\overline{k}+m+1) w_{k, \mu_2+1, m+1}\;.
\end{split}
\end{align}

\subsection{$\mathfrak{sl}_3$-action in $M^{\vee}_{\overline{l} \in \{0, 1\}}$}
In $M^{\vee}_{\overline{l} \in \{0, 1\}}$ \eqref{eq:actiononetaabc}
reduces to 
\begin{align}\label{eq:actiononetamu2andmu2+1}
\begin{split}
e_1 \eta_{k, \mu_2, m} &= -(\overline{k}+m-1) \eta_{k-1,\mu_2,m}+\begin{cases}\eta_{k, \mu_2+1, m-1} & m > 0 \\ 0 & m=0\end{cases} \\
e_1 \eta_{k, \mu_2+1, m} &= -\frac{(\overline{k}-2)(\overline{k}+m)}{\overline{k}} \eta_{k-1, \mu_2+1, m}\\
e_2 \eta_{k, \mu_2, m}&= -\begin{cases} (\overline{k}+1) \eta_{k+1, \mu_2, m-1} & m > 0 \\0 & m = 0\end{cases} \\
e_2 \eta_{k, \mu_2+1,m} &= - \begin{cases}(\overline{k}+1) \eta_{k+1, \mu_2+1,m-1}& m > 0 \\0 & m=0\end{cases}\\
f_1 \eta_{k, \mu_2, m} &= (\overline{k}+1)\eta_{k+1, \mu_2, m} \\
f_1 \eta_{k, \mu_2+1, m} &= (\overline{k}+1)\eta_{k+1, \mu_2+1,m} \\
f_2 \eta_{k, \mu_2, m} &= \eta_{k, \mu_2+1,m}-(m+1)\eta_{k-1, \mu_2, m+1} \\
f_2 \eta_{k, \mu_2+1,m} &= -(m+1)\frac{\overline{k}-2}{\overline{k}} \eta_{k-1,\mu_2+1,m+1} \\
e_{12}\eta_{k, \mu_2, m} &= \begin{cases}(\overline{k}+m-1)\eta_{k, \mu_2, m-1} & m > 0 \\ 0 & m=0\end{cases}\\
e_{12} \eta_{k, \mu_2+1, m} &= \begin{cases}(\overline{k}+m) \eta_{k, \mu_2+1, m-1} & m > 0 \\ 0 & m=0\end{cases} \\
f_{12} \eta_{k, \mu_2, m} &= -(m+1) \eta_{k, \mu_2, m+1}\\
f_{12}\eta_{k, \mu_2+1, m} &= -(m+1) \eta_{k, \mu_2+1, m+1}\;.
\end{split}
\end{align}

\subsection{Structure of $M_{\overline{l} \geq 2}$, $M_{\overline{l} \geq 1}$, $M_{\overline{l} \geq 0}$, $M_{\overline{l} = 0}$, $M_{\overline{l}=1}$, $M_{\overline{l} \in \{0, 1\}}$}

The proof of the following proposition is similar to the one of Theorem \autoref{Thm:Mmu1mu2simplicitycriterion}. 

 \begin{Prop}\label{Prop:Mlgeqmu2+2eqmu2eqmu2+1simple}
 $M_{\overline{l} \geq 2}$, $M_{\overline{l} = 0}$, $M_{\overline{l}= 1}$  are simple $\mathfrak{sl}_3$-modules. 
 \end{Prop}
 
 \begin{proof}
 Let $S$ be a $\mathfrak{sl}_3$-submodule of the module under consideration. \\
(1) We look at \eqref{eq:actionofbasisofsl3onwklm}.
We assume for some $k_0 \in \Z, \overline{l_0} \geq 2, m_0 \geq 0$ that $w_{k_0, l_0, m_0} \in S$. We act with $e_{12}$ and find $w_{k_0, l_0, 0}\in S$.
We then act with $e_1$ and $e_2$ to find $w_{k,l,0} \in S$, $k \leq k_0, 2 \leq \overline{l} \leq \overline{l_0}$.  
We then act with $f_{12}$ to find $w_{k,l,m} \in S$, $k \leq k_0, 2 \leq \overline{l} \leq \overline{l_0}, m \geq 0$. We have for $\overline{l} \geq 2$
\begin{align*}
C w_{k+1,l,m} = f_1 w_{k,l,m}+ \begin{cases}\overline{l}w_{k,l-1,m+1} & \overline{l} \geq 3 \\ 0 & \overline{l} = 2\end{cases}
\end{align*}
with $C \neq 0$ and hence $w_{k_0+1,l,m}, w_{k_0+2,l,m}, \dots \in S$. Thus $w_{k,l,m} \in S$, $k \in \Z, 2 \leq \overline{l} \leq \overline{l_0}, m \geq 0$. 
Finally, we have for $\overline{l} \geq 2$
\begin{align*}
C w_{k, l+1, m}= f_2 w_{k,l,m}-\overline{k} w_{k-1,l, m+1}
\end{align*}
with $C \neq 0$ and hence $w_{k,l_0+1,m}, w_{k, l_0+2, m}, \dots \in S$. Thus $w_{k,l,m} \in S$, $k \in \Z, \overline{l} \geq 2, m \geq 0$. 
Thus $S=M_{\overline{l} \geq 2}$. \\

(2) We find the $\mathfrak{sl}_3$-action in $M_{\overline{l}=0}$
from \eqref{eq:actiononwmu2andmu2+1}
\begin{align*}
\begin{split}
e_1 w_{k, \mu_2, m} &= -\overline{k} w_{k-1, \mu_2, m} \\
e_2 w_{k, \mu_2, m} &= m w_{k+1, \mu_2, m-1} \\
f_1 w_{k, \mu_2, m} &= (\overline{k}+m)w_{k+1, \mu_2, m}\\
f_2 w_{k, \mu_2, m} &= \overline{k}w_{k-1, \mu_2, m+1}\\
e_{12} w_{k, \mu_2, m} &= -m w_{k, \mu_2, m-1} \\
f_{12} w_{k, \mu_2, m} &= (\overline{k}+m)w_{k, \mu_2, m+1}\;.
\end{split}
\end{align*}
Let $w_{k_0, \mu_2, m_0} \in S$. Acting with $e_{12}$ and $f_{12}$ we find $w_{k_0,\mu_2, m} \in S$, $m \geq 0$.
Acting with $f_1$ we find $w_{k, \mu_2, m} \in S$, $k \geq k_0, m \geq 0$. Acting with $e_1$ we find $w_{k,\mu_2, m} \in S$,
$k \in \Z, m \geq 0$. \\

(3) We find the $\mathfrak{sl}_3$-action in $M_{\overline{l}=1}$ from \eqref{eq:actiononwmu2andmu2+1}
\begin{align*}
e_1 w_{k, \mu_2+1,m} &= -\overline{k} w_{k-1, \mu_2+1,m}\\
e_2 w_{k, \mu_2+1,m} &= m \frac{\overline{k}-1}{\overline{k}+1}w_{k+1,\mu_2+1,m-1}\\
f_1 w_{k, \mu_2+1,m} &= \frac{(\overline{k}-1)(\overline{k}+m+1)}{\overline{k}+1}w_{k+1, \mu_2+1,m}\\
f_2 w_{k, \mu_2+1,m} &= \overline{k} w_{k-1, \mu_2+1, m+1}\\
e_{12} w_{k, \mu_2+1,m} &= -m w_{k, \mu_2+1, m-1}\\
f_{12} w_{k, \mu_2+1, m} &= (\overline{k}+m+1)w_{k, \mu_2+1,m+1}\;.
\end{align*}
Let $w_{k_0, \mu_2+1, m_0} \in S$. Acting with $e_{12}$ and $f_{12}$ we find $w_{k_0, \mu_2+1, m} \in S$, $m \geq 0$.
Acting with $f_1$ we find $w_{k, \mu_2+1,m} \in S$, $k \geq k_0, m \geq 0$. Acting with $e_1$ we find $w_{k, \mu_2+1, m} \in S$, $k \in \Z, m \geq 0$. 
 \end{proof}

\begin{Prop}\label{Prop:structureofMbgeqmu2etc}

Given a homomorphism $\phi$ respectively $\psi$ of $\mathfrak{sl}_3$-modules we associate $x_{k,l,m} \in \C$ defined by $\phi(\eta_{k,l,m}) = x_{k,l,m} w_{k,l,m}$ respectively $\psi(w_{k,l,m}) = x_{k,l,m} \eta_{k,l,m}$ if $\eta_{k,l,m}$ respectively $w_{k,l,m}$ form a basis
of the domain of $\phi$ respectively $\psi$. 

\begin{enumerate}

\item 
We have an isomorphism 
\begin{align*}
\Hom(M_{\overline{l} \in \{0, 1\}}^{\vee}, M_{\overline{l} \in \{0, 1\}}) \xrightarrow{\cong} \C,\; \phi \mapsto x_{0, \mu_2, 0}\;.
\end{align*}
If $\phi \neq 0$ we have $\Imag \phi = M_{\overline{l}=0}$, $\ker \phi = M^{\vee}_{\overline{l}=1}$,
and $\phi$ induces an isomorphism $M^{\vee}_{\overline{l}=0} \cong M_{\overline{l}=0}$ of $\mathfrak{sl}_3$-modules. 
We have an isomorphism 
\begin{align*}
\Hom(M_{\overline{l} \in \{0, 1\}}, M_{\overline{l} \in \{0, 1\}}^{\vee}) \xrightarrow{\cong} \C,\; \psi \mapsto x_{0, \mu_2+1, 0}\;.
\end{align*}
If $\psi \neq 0$ we have $\Imag \psi = M_{\overline{l}=1}^{\vee}$, $\ker \psi = M_{\overline{l}=0}$, 
and $\psi$ induces an isomorphism $M_{\overline{l}=1} \cong M^{\vee}_{\overline{l}=1}$ of $\mathfrak{sl}_3$-modules.  

\item
We have an isomorphism
\begin{align*}
\Hom(M_{\overline{l} \geq 1}^{\vee}, M_{\overline{l} \geq 1}) \xrightarrow{\cong} \C,\; \phi \mapsto x_{0, \mu_2+1,0}\;.
\end{align*}
For $\phi \neq 0$ we have $\Imag \phi = M_{\overline{l} = 1}$,
$\ker \phi = M^{\vee}_{\overline{l} \geq 2}$,  and $\phi$ induces an $M^{\vee}_{\overline{l}=1} \cong M_{\overline{l}=1}$. We have an isomorphism
\begin{align*}
\Hom(M_{\overline{l} \geq 1}, M_{\overline{l} \geq 1}^{\vee})
\xrightarrow{\cong} \C,\; \psi \mapsto x_{0, \mu_2+2, 0}\;.
\end{align*}
For $\psi \neq 0$ we have $\Imag \psi = M_{\overline{l} \geq 2}^{\vee}$,
$\ker \psi = M_{\overline{l}=1}$, and $\psi$ induces an isomorphism $M_{\overline{l} \geq 2} \cong M^{\vee}_{\overline{l} \geq 2}$. 

\item We have an isomorphism
\begin{align*}
\Hom(M_{\overline{l} \geq 0}^{\vee}, M_{\overline{l} \geq 0}) \xrightarrow{\cong} \C,\; \phi \mapsto x_{0, \mu_2,0}\;.
\end{align*}
For $\phi \neq 0$ we have $\Imag \phi = M_{\overline{l} = 0}$, $\ker \phi = M^{\vee}_{\overline{l} \geq 1}$,  and $\phi$ induces an isomorphism $M^{\vee}_{\overline{l}=0} \cong M_{\overline{l}=0}$. We have
an isomorphism
\begin{align*}
\Hom(M_{\overline{l} \geq 0}, M_{\overline{l} \geq 0}^{\vee})
\xrightarrow{\cong} \C,\; \psi \mapsto x_{0, \mu_2+2, 0}\;.
\end{align*}
For $\psi \neq 0$ we have $\Imag \psi = M_{\overline{l} \geq 2}^{\vee}$,
$\ker \psi = M_{\overline{l} \in \{0, 1\}}$,
and $\psi$ induces an isomorphism $M_{\overline{l} \geq 2} \cong M_{\overline{l} \geq 2}^{\vee}$. 
\end{enumerate}

\end{Prop}

We note that from (1) it is clear that $M_{\overline{l} \in \{0, 1\}} \ncong M^{\vee}_{\overline{l} \in \{0, 1\}}$,
and from (2) and (3) we conclude the analogous statement for $M_{\overline{l} \geq 1}$ and
$M_{\overline{l} \geq 0}$. 

\begin{proof}
(1) Let $\phi: M^{\vee}_{\overline{l} \in \{0, 1\}} \rightarrow M_{\overline{l} \in \{0, 1\}}$. 
We deduce from \eqref{eq:actiononwmu2andmu2+1} and \eqref{eq:actiononetamu2andmu2+1} with $k \in \Z$ and, 
unless
stated otherwise, $m \geq 0$ 

\begin{align}
\begin{split}\label{eq:equationsforthex}
e_1 \phi(\eta_{k, \mu_2, m}):\; & -\overline{k}x_{k, \mu_2, m}= -(\overline{k}+m-1)x_{k-1, \mu_2, m}\quad 0 = x_{k, \mu_2+1, m-1}, m>0\\
e_1 \phi(\eta_{k, \mu_2+1, m}):\;  & -\overline{k} x_{k, \mu_2+1,m} = -\frac{(\overline{k}-2)(\overline{k}+m)}{\overline{k}} x_{k-1, \mu_2+1,m}\; \\
e_2 \phi(\eta_{k, \mu_2, m}):\; & m x_{k, \mu_2, m} = -(\overline{k}+1)x_{k+1, \mu_2, m-1}, m > 0 \\
e_2 \phi(\eta_{k, \mu_2+1,m}):\; & -x_{k, \mu_2+1,m} = 0\quad -(\overline{k}+1)x_{k+1, \mu_2+1, m-1} = m \frac{\overline{k}-1}{\overline{k}+1} x_{k, \mu_2+1,m}, m > 0\\
f_1 \phi(\eta_{k, \mu_2, m}):\; & (\overline{k}+m)x_{k, \mu_2, m} = (\overline{k}+1)x_{k+1, \mu_2, m} \\
f_1 \phi(\eta_{k, \mu_2+1, m}):\; & \frac{(\overline{k}-1)(\overline{k}+m+1)}{\overline{k}+1} x_{k, \mu_2+1,m} = (\overline{k}+1)x_{k+1, \mu_2+1,m} \quad -x_{k, \mu_2+1, m} = 0\\
f_2 \phi(\eta_{k, \mu_2, m}):\; & \overline{k} x_{k, \mu_2, m} = -(m+1)x_{k-1, \mu_2, m+1}\quad
x_{k, \mu_2+1, m} = 0 \\
f_2 \phi(\eta_{k, \mu_2+1,m}):\; &
\overline{k}x_{k, \mu_2+1, m} = -(m+1) \frac{\overline{k}-2}{\overline{k}} x_{k-1, \mu_2+1, m+1} \\
e_{12} \phi(\eta_{k, \mu_2, m}):\; & -m x_{k, \mu_2, m} = (\overline{k}+m-1)x_{k, \mu_2, m-1}, m > 0\\
e_{12}\phi(\eta_{k, \mu_2+1,m}):\; & -m x_{k, \mu_2+1,m} = (\overline{a}+m)x_{k, \mu_2+1, m-1}, m > 0 \\
f_{12}\phi(\eta_{k, \mu_2,m}):\; & (\overline{k}+m) x_{k, \mu_2, m} = -(m+1) x_{k, \mu_2, m+1}\\
f_{12} \phi(\eta_{k, \mu_2+1,m}):\; & (\overline{k}+m+1)x_{k, \mu_2+1,m} = -(m+1) x_{k, \mu_2+1, m+1}\;.
\end{split}
\end{align}

From the first line we see that $x_{k, \mu_2+1, m} = 0$ for $k \in \Z, m \geq 0$ and it is easy to show that for 
$k \in \Z$
\begin{align}\label{eq:solutionxamu2c}
x_{k, \mu_2,m} = -\mu_1 x_{0, \mu_2, 0}\begin{cases}\frac{1}{\overline{k}} & m = 0 \\ (-1)^m \frac{(\overline{k}+1)^{(m-1)}}{m!} & m \geq 1 \end{cases}
\end{align} 
is the unique solution of \eqref{eq:equationsforthex} with $x_{0, \mu_2, 0} \in \C$ arbitrary. 
This proves the statement about $\Hom(M_{\overline{l} \in \{0, 1\}}^{\vee}, M_{\overline{l} \in \{0, 1\}})$. \par

We now consider a map in the opposite direction
$\psi: M_{\overline{l} \in \{0, 1\}} \rightarrow M_{\overline{l} \in \{0, 1\}}^{\vee}$.
We proceed completely parallely. 
From \eqref{eq:actiononwmu2andmu2+1} and \eqref{eq:actiononetamu2andmu2+1} we get
\begin{align}\label{eq:equationforthexformapfromMoMvee}
\begin{split}
e_1 \psi(w_{k, \mu_2, m}):\; & -(\overline{k}+m-1)x_{k, \mu_2, m} =-\overline{a} x_{k-1, \mu_2, m}  \quad x_{k, \mu_2, m} = 0, m > 0 \\
e_1 \psi(w_{k, \mu_2+1, m}):\; & -\frac{(\overline{k}-2)(\overline{k}+m)}{\overline{k}} x_{k, \mu_2+1,m} = -\overline{k}x_{k-1, \mu_2+1,m}\\
e_2 \psi(w_{k, \mu_2, m}):\; & -(\overline{k}+1)x_{k, \mu_2, m} = m x_{k+1, \mu_2, m-1},\; m > 0\\
e_2 \psi(w_{k, \mu_2+1,m}):\; & -x_{k,\mu_2,m} = 0 \quad -(\overline{k}+1) x_{k, \mu_2+1,m} = m \frac{\overline{k}-1}{\overline{k}+1}x_{k+1, \mu_2+1, m-1},\; m > 0 \\
f_1 \psi(w_{k, \mu_2, m}):\; & (\overline{k}+1) x_{k, \mu_2, m} = (\overline{k}+m)x_{k+1, \mu_2, m}\\
f_1 \psi(w_{k, \mu_2+1,m}):\; & (\overline{k}+1)x_{k, \mu_2+1,m} = \frac{(\overline{k}-1)(\overline{k}+m+1)}{\overline{k}+1}x_{k+1, \mu_2+1,m}\quad x_{k, \mu_2, m+1} = 0 \\
f_2 \psi(w_{k, \mu_2, m}):\; & x_{k, \mu_2, m} = 0 \quad -(m+1)x_{k, \mu_2,m} = \overline{k} x_{k-1, \mu_2, m+1} \\
f_2 \psi(w_{k,\mu_2+1,m}):\; & -(m+1) \frac{\overline{k}-2}{\overline{k}} x_{k, \mu_2+1, m} = \overline{k} x_{k-1, \mu_2+1, k+1} \\
e_{12} \psi(w_{k, \mu_2, m}):\; & (\overline{k}+m-1) x_{k, \mu_2, m} = -m x_{k, \mu_2, m-1}, m > 0\\
e_{12} \psi(w_{k, \mu_2+1,m}):\; & (\overline{k}+m)x_{k, \mu_2+1,m} = -m x_{k, \mu_2+1,m-1}, m >0\\
f_{12}\psi(w_{k, \mu_2, m}):\; & -(m+1)x_{k, \mu_2, m} = (\overline{k}+m)x_{k, \mu_2, m+1} \\
f_{12}\psi(w_{k, \mu_2+1,m}):\; & (\overline{k}+m+1)x_{k, \mu_2+1,m+1} = -(m+1)x_{k, \mu_2+1,m} 
\end{split}
\end{align}
From $e_2 \psi(w_{k, \mu_2+1, m})$ we see that $x_{k, \mu_2, m} = 0$ for $k \in \Z, m \geq 0$ and it is easy to show that
for $k \in \Z, m \geq 0$ 
\begin{align} \label{eq:solutionxamu2+1c}
x_{k, \mu_2+1,m} = (-1)^m \frac{m!}{(\overline{k}+1)^{(m)}} \frac{\overline{k}(\overline{k}-1)}{\mu_1 (\mu_1+1)} x_{0, \mu_2+1,0}
\end{align}
is the unique solution to \eqref{eq:equationforthexformapfromMoMvee}, where $x_{0, \mu_2+1,0} \in \C$ is arbitrary.  \\

(2) We consider $\phi: M^{\vee}_{\overline{l} \geq 1} \rightarrow M_{\overline{l} \geq 1}$. 
From \eqref{eq:actionofbasisofsl3onwklm} and \eqref{eq:actiononetaabc} we see that the coefficients of $f_2\phi(\eta_{k,l,m})$ are given by
\begin{align}\label{eq:coeffoff2phietaabc}
\begin{split}
& \frac{\overline{l}(\overline{l}-1)(\overline{k}+\overline{l}+m)}{(\overline{k}+\overline{l})(\overline{k}+\overline{l}-1)}x_{k,l,m}
= (\overline{l}+1)x_{k,l+1,m}\\
& \overline{k}x_{k,l,m} = -(m+1) \frac{(\overline{k}-1)(\overline{k}-2)}{(\overline{k}+\overline{l}-1)(\overline{k}+\overline{l}-2)}x_{k-1,l,m+1}\;.
\end{split}
\end{align}  
The first equation implies $x_{k,l,m} = 0$ for $k \in \Z, \overline{l} \geq 2, m \geq 0$, and hence the map factors through
\begin{align*}
\xymatrix{ M_{\overline{l} \geq 1}^{\vee} \ar[r]^{\phi}\ar[d] & M_{\overline{l} \geq 1} \\
M_{\overline{l}=1}^{\vee}\ar[r]^{\overline{\phi}} & M_{\overline{l}=1} \ar[u]
}\;.
\end{align*}
Here the vertical maps are the quotient and inclusion map. 
The map $\overline{\phi}$ is determined as part of the analysis of the proof of (1), see \eqref{eq:solutionxamu2+1c}, hence the statement.  \par

We now consider a map in the opposite direction $\psi: M_{\overline{l} \geq 1}
\rightarrow M^{\vee}_{\overline{l} \geq 1}$. 
From \eqref{eq:actionofbasisofsl3onwklm} and \eqref{eq:actiononetaabc} we see that for $k \in \Z, m \geq 0$ the coefficients of $f_2 \psi(w_{k, \mu_2+1, m})$ are given by

\begin{align*}
x_{k, \mu_2+1,m} = 0\quad -(m+1) \frac{\overline{k}-2}{\overline{k}}x_{k, \mu_2+1, m} = \overline{k} x_{k-1, \mu_2+1,m+1} 
\end{align*}
Hence $x_{k, \mu_2+1,m} = 0$ for $k \in \Z, m \in \Z_{\geq 0}$, and the map $\psi$ factors through 
\begin{align*}
\xymatrix{ M_{\overline{l} \geq 1} \ar[d] \ar[r]^{\psi} & M^{\vee}_{\overline{l} \geq 1} \\
M_{\overline{l} \geq 2} \ar[r]^{\overline{\psi}} & M^{\vee}_{\overline{l} \geq 2} \ar[u]}\;.
\end{align*}
Here the vertical maps are the quotient and inclusion map. 
By a similar analysis as in the proof of (1) the map $\overline{\psi}$ is uniquely determined by $x_{0, \mu_2+2, 0} \in \C$, namely we have for $k \in \Z, \overline{l} \geq 2, m \geq 0$

\begin{align}\label{eq:xklmforlgemu2+2}
x_{k,l,m} = \frac{(-1)^m m!}{(\overline{k}+\overline{l})^{(m)}}\; \; \left(-\frac{\overline{k}}{\mu_1}\right) x_{0, \mu_2+2,0} \begin{cases}\frac{(\overline{l}-\mu_1-1)^{(k)}}{(-\mu_1-1)^{(k)}} & k \geq 0 \\ \frac{(-\mu_1-2)_{(-k)}}{(\overline{l}-\mu_1-2)_{(-k)}} & k < 0 \end{cases} \times \begin{cases} \frac{\overline{l}\; (-\mu_1+1)^{(\overline{l}-2)}}{2(\overline{l}-2)(2 \cdots (\overline{l}-3))} & \overline{l} \geq 3 \\ 1 & \overline{l}=2\end{cases}\;.
\end{align}
In the case $\overline{l} \geq 3$  we understand $(2 \cdots (\overline{l}-3)) = 1$ if $\overline{l} < 5$.
\\

(3) We consider $\phi: M_{\overline{l} \geq 0}^{\vee}\rightarrow M_{\overline{l} \geq 0}$. 
The following argument is analogous to the one for the corresponding map in (2).  
Namely, the first equation in \eqref{eq:coeffoff2phietaabc} implies $x_{k, l, m} = 0$ for $\overline{l} \geq 1$.
Hence $\phi$ factors through
\begin{align*}
\xymatrix{ M_{\overline{l} \geq 0}^{\vee} \ar[r]^{\phi} \ar[d] & M_{\overline{l} \geq 0} \\
M_{\overline{l}=0}^{\vee} \ar[r]^{\overline{\phi}} & M_{\overline{l} = 0} \ar[u]}\;.
\end{align*}
The map $\overline{\phi}$ is determined as part of the analysis of the proof of (1), see \eqref{eq:solutionxamu2c}. \par

 We now consider $\psi: M_{\overline{l} \geq 0} \rightarrow M^{\vee}_{\overline{l} \geq 0}$. 
 The following argument is analogous to the one for the corresponding map in (2). 
 For $k \in \Z, m \geq 0$ we see from \eqref{eq:actionofbasisofsl3onwklm} and \eqref{eq:actiononetaabc} that the coefficients of $f_2 \psi(w_{k, \mu_2, m})$ and
 $f_2 \psi(w_{k, \mu_2+1, m})$ are given by
\begin{align*}
& x_{k, \mu_2, m} = 0\qquad -(m+1) x_{k, \mu_2, m} = \overline{k} x_{k-1, \mu_2, m+1}\\
& x_{k, \mu_2+1, m}= 0\qquad -(m+1) \frac{\overline{k}-2}{\overline{k}} x_{k, \mu_2+1, m} = \overline{k} x_{k-1, \mu_2+1, m+1}\;.
\end{align*}
Hence the map $\psi$ factors through
\begin{align*}
\xymatrix{ M_{\overline{l} \geq 0} \ar[d] \ar[r]^{\psi} & M^{\vee}_{\overline{l} \geq 0} \\
M_{\overline{l} \geq 2} \ar[r]^{\overline{\psi}} & M^{\vee}_{\overline{l} \geq 2} \ar[u]}\;.
\end{align*}
The map $\overline{\psi}$ again satisfies \eqref{eq:xklmforlgemu2+2}
and in particular is uniquely determined by $x_{0, \mu_2+2, 0} \in \C$. 
\end{proof}

\subsubsection{Identification with relaxed highest weight $\mathfrak{sl}_3$-modules}

For $\lambda \in \mathfrak{t}^*$ and $\mu \in \C$ we consider the relaxed Verma 
$\mathfrak{sl}_3$-module 

\begin{align*}
\RVerma_{\mathfrak{p}_1,\lambda,\mu} = \Uea \mathfrak{sl}_3 \otimes_{\Uea \mathfrak{p}_1}(\RVerma^{\mathfrak{sl}_2^{(1)}}_{\lambda(h_1), \mu} \otimes \C_{\lambda})
\end{align*}

as defined e.g. in \cite{Eic16b}.
Here we use the parabolic subalgebra $\mathfrak{p}_1 = \C f_1 \oplus \mathfrak{b}$ and
the $\mathfrak{sl}_2$-subalgebra $\mathfrak{sl}_2^{(1)} = \C f_1 \oplus
\C h_1 \oplus \C e_1$ of $\mathfrak{sl}_3$. 
We recall in particular the relations 
\begin{align*}
h (1 \otimes 1_{\lambda}) = (\lambda+\mu \alpha_1)(h) 1 \otimes 1_{\lambda},\; h \in\mathfrak{t}
\end{align*}
and 
\begin{align*}
f_1 e_1 (1 \otimes 1_{\lambda}) = -(\mu+1)(\mu+\lambda(h_1)) 1 \otimes 1_{\lambda}
\end{align*}
for the generating vector
$1 \otimes 1_{\lambda}$ of $\RVerma_{\mathfrak{p}_1, \lambda, \mu}$. 
$\RVerma_{\mathfrak{p}_1,\lambda,\mu}$ has a unique simple quotient
denoted by $\Lsimple_{\mathfrak{p}_1, \lambda, \mu}$. 

\begin{Prop} \label{Prop:subquotsofMasrelaxedVerma}

\begin{enumerate}
\item $M_{\mu_1, \mu_2, \overline{l} \geq 0}\cong \RVerma_{\mathfrak{p}_1, 0, \mu_1}^{\vee}$

\item $M_{\mu_1, \mu_2, \overline{l} \geq 1} \cong \RVerma_{\mathfrak{p}_1, -\rho, \mu_1}^{\vee}$

\item $M_{\mu_1, \mu_2, \overline{l} \geq 2} \cong \RVerma_{\mathfrak{p}_1, -2\alpha_2, \mu_1}
= \Lsimple_{\mathfrak{p}_1, -2\alpha_2, \mu_1}$ and $\RVerma_{\mathfrak{p}_1, -2\alpha_2, \mu_1} \cong \RVerma^{\vee}_{\mathfrak{p}_1, -2\alpha_2, \mu_1}$

\item $M_{\mu_1, \mu_2, \overline{l} = 0}\cong \Lsimple_{\mathfrak{p}_1, 0, \mu_1}$
and $\Lsimple_{\mathfrak{p}_1, 0, \mu_1}= \RVerma_{\mathfrak{p}_1, 0, \mu_1}/
\RVerma_{\mathfrak{p}_1, -\rho, \mu_1}$ and $\Lsimple_{\mathfrak{p}_1, 0, \mu_1} 
\cong \Lsimple^{\vee}_{\mathfrak{p}_1, 0, \mu_1}$

\item $M_{\mu_1, \mu_2, \overline{l}=1} \cong \Lsimple_{\mathfrak{p}_1, -\rho, \mu_1}$
and $\Lsimple_{\mathfrak{p}_1, -\rho, \mu_1}= \RVerma_{\mathfrak{p}_1, -\rho, \mu_1}/
\RVerma_{\mathfrak{p}_1, - 2\alpha_2, \mu_1}$ and $\Lsimple_{\mathfrak{p}_1, -\rho, \mu_1} 
\cong \Lsimple^{\vee}_{\mathfrak{p}_1, -\rho, \mu_1}$
\end{enumerate}

Here $\cong$ denote isomorphisms of $\mathfrak{sl}_3$-modules. 

\end{Prop}

We note that by  (4) and (5) and \eqref{eq:actiononwmu2andmu2+1} we obtain a non-split exact sequence of $\mathfrak{sl}_3$-modules 

\begin{align*}
0 \rightarrow \Lsimple_{\mathfrak{p}_1,0,\mu_1} \rightarrow M_{\mu_1, \mu_2, \overline{l} \in \{0, 1\}} \rightarrow \Lsimple_{\mathfrak{p}_1, -\rho, \mu_1} \rightarrow 0\;.
\end{align*}

\begin{proof}
(1)
By Remark \autoref{Rem:subquotientsofMmu1mu2} $\MVerma^{\vee}_{\overline{l} \geq 0}$ is a quotient of $\MVerma^{\vee}$. 
We construct an isomorphism of $\mathfrak{sl}_3$-modules
$\RVerma_{\mathfrak{p}_1, 0, \mu_1} \xrightarrow{\cong} \MVerma^{\vee}_{\overline{l} \geq 0}$. 
In $M^{\vee}_{\overline{l} \geq 0}$ we have by \eqref{eq:actiononetamu2andmu2+1}
for $k \in \Z$

\begin{align}\label{eq:actiononetaamu20}
\begin{split}
e_1 \eta_{k,\mu_2,0} &= -(\overline{k}-1)\eta_{k-1,\mu_2,0} \\
e_2 \eta_{k,\mu_2,0}&= 0 \\
f_1 \eta_{k,\mu_2,0} &= (\overline{k}+1)\eta_{k+1,\mu_2,0} \\
f_2 \eta_{k,\mu_2,0} &= \eta_{k,\mu_2+1,0}-\eta_{k-1,\mu_2,1} \\
e_{12} \eta_{k,\mu_2,0} &= 0 \\
f_{12} \eta_{k,\mu_2,0} &= -\eta_{k,\mu_2,1}\;.
\end{split}
\end{align}

We recall that the $\mathfrak{t}$-weight of $\eta_{k, l, m}$ is \eqref{eq:weightofuklm},
hence $h \eta_{0, \mu_2, 0} = \mu_1 \alpha_1(h) \eta_{0, \mu_2, 0}$,
$h \in \mathfrak{t}$ .
\eqref{eq:actiononetaamu20} shows that 
we have an isomorphism of simple $\mathfrak{p}_1$-modules
\begin{align*}
\RVerma^{\mathfrak{sl}_2^{(1)}}_{0,\mu_1} \otimes \C_0 \xrightarrow{\cong}\bigoplus_{k \in \Z}\C \eta_{k, \mu_2, 0},\; 1 \otimes 1_0 \mapsto \eta_{0,\mu_2,0}\;. 
\end{align*}
The isomorphism induces a homomorphism of $\mathfrak{sl}_3$-modules
$\phi: \RVerma_{\mathfrak{p}_1, 0, \mu_1}
\rightarrow M^{\vee}_{\overline{l} \geq 0}$.
The image contains $\eta_{k,\mu_2,0}$, $k \in \Z$, and by application of $f_{12}$ also $\eta_{k, \mu_2, m}$ for $k \in \Z, m \geq 0$. 
We now prove by induction on  $\overline{l^{\prime}} \geq 0$ that $\eta_{k, l, m}$ is in the image for $k \in \Z, m \geq 0, 0 \leq \overline{l} \leq \overline{l^{\prime}}$.
The induction start was given. The induction step follows from 
\begin{align*}
(\overline{l^{\prime}}+1)\eta_{k,l^{\prime}+1,m} = f_2 \eta_{k,l^{\prime},m} + (m+1) \frac{(\overline{k}-1)(\overline{k}-2)}{(\overline{k}+\overline{l^{\prime}}-1)(\overline{k}+\overline{l^{\prime}}-2)}\eta_{k-1,l^{\prime},m+1} 
\end{align*}
since $\overline{l^{\prime}}+1 > 0$. Thus $\phi$ surjects. 
Since the formal character of $\RVerma_{\mathfrak{p}_1, 0, \mu_1}$ 
and of $M^{\vee}_{\overline{l} \geq 0}$ is easily computed to be 
$\frac{\sum_{n \in \Z} e^{(n+\mu_1)\alpha_1}}{(1-e^{-\alpha_1-\alpha_2})(1-e^{-\alpha_2})}$,
 $\phi$ is an isomorphism. 
This proves the statement. \\

(2) By Remark \autoref{Rem:subquotientsofMmu1mu2} $\MVerma^{\vee}_{\overline{l} \geq 1}$ is a submodule of $\MVerma^{\vee}$,
and also of $\MVerma^{\vee}_{\overline{l} \geq 0}$. 
In $M^{\vee}_{\overline{l} \geq 1}$ we have by \eqref{eq:actiononetamu2andmu2+1}
for $k \in \Z$
\begin{align*}%\label{eq:actionofetaamu2+10}
\begin{split}
e_1 \eta_{k,\mu_2+1,0} &= -(\overline{k}-2)\eta_{k-1,\mu_2+1,0} \\
e_2 \eta_{k,\mu_2+1,0} &= 0\\
f_1 \eta_{k,\mu_2+1,0} &= (\overline{k}+1)\eta_{k+1,\mu_2+1,0} \\
f_2 \eta_{k,\mu_2+1,0} &= 2 \eta_{k,\mu_2+2,0}-\frac{\overline{k}-2}{\overline{k}} \eta_{k-1,\mu_2+1,1} \\
e_{12} \eta_{k,\mu_2+1,0} &= 0 \\
f_{12} \eta_{k,\mu_2+1,0} &= -\eta_{k,\mu_2+1,1}
\end{split}
\end{align*}
and hence
an isomorphism
of simple $\mathfrak{p}_1$-modules 
\begin{align*}
\RVerma^{\mathfrak{sl}_2^{(1)}}_{-1,\mu_1} \otimes\C_{-\rho}\xrightarrow{\cong}\bigoplus_{k \in \Z} \C \eta_{k, \mu_2+1,0},\; 1 \otimes 1_{-\rho}\mapsto 
\eta_{1,\mu_2+1,0}\;,
\end{align*}
which induces an isomorphism $\RVerma_{\mathfrak{p}_1,-\rho,\mu_1}\xrightarrow{\cong} M^{\vee}_{\overline{l} \geq 1}$ of $\mathfrak{sl}_3$-modules
by the same argument as in (1).\\

(3) By Remark \autoref{Rem:subquotientsofMmu1mu2} $\MVerma^{\vee}_{\overline{l} \geq 2}$ is a submodule of $\MVerma^{\vee}$,
and also of $\MVerma^{\vee}_{\overline{l} \geq 1}$.  In $M^{\vee}_{\overline{l} \geq 2}$ we have by \eqref{eq:actiononetaabc} 
\begin{align*}%\label{eq:actionofetaamu2+20}
\begin{split}
e_1 \eta_{k, \mu_2+2,0} &= -\frac{(\overline{k}-1)(\overline{k}-2)}{\overline{k}} \eta_{k-1,\mu_2+2,0} \\
e_2 \eta_{k, \mu_2+2,0}&= 0\\
f_1 \eta_{k, \mu_2+2,0} &= (\overline{k}+1)\eta_{k+1,\mu_2+2,0} \\
f_2 \eta_{k, \mu_2+2,0} &= 3 \eta_{k, \mu_2+3,0}-\frac{(\overline{k}-1)(\overline{k}-2)}{\overline{k}(\overline{k}+1) }\eta_{k-1, \mu_2+2,1}\\
e_{12} \eta_{k, \mu_2+2,0} & = 0\\
f_{12} \eta_{k, \mu_2+2,0} &= -\eta_{k, \mu_2+2, 1} 
\end{split}
\end{align*}
and hence an isomorphism of simple $\mathfrak{p}_1$-modules
\begin{align*}
\RVerma^{\mathfrak{sl}_2^{(1)}}_{2,\mu_1} \otimes\C_{-2\alpha_2}\xrightarrow{\cong}\bigoplus_{k \in \Z} \C \eta_{k, \mu_2+2,0},\; 1 \otimes 1_{-2\alpha_2} \mapsto 
\eta_{0,\mu_2+2,0}\;,
\end{align*}
which induces an isomorphism 
$\RVerma_{\mathfrak{p}_1,-2\alpha_2,\mu_1}\xrightarrow{\cong} M^{\vee}_{\overline{l} \geq 2}$ by the same argument as in (1). 
Since $M_{\overline{l} \geq 2}$ is
self-dual by Proposition \autoref{Prop:structureofMbgeqmu2etc}(2), we conclude that $\RVerma_{\mathfrak{p}_1, -2\alpha_2, \mu_1} \cong M_{\overline{l} \geq 2}$ and that $\RVerma_{\mathfrak{p}_1, -2\alpha_2, \mu_1}$
is also self-dual. Moreover, by Proposition \autoref{Prop:Mlgeqmu2+2eqmu2eqmu2+1simple}
we have $\RVerma_{\mathfrak{p}_1, -2\alpha_2, \mu_1} = \Lsimple_{\mathfrak{p}_1, -2\alpha_2, \mu_1}$.
\\

(4)
(1) and (2) induce an isomorphism of $\mathfrak{sl}_3$-modules
\begin{align*}
\RVerma_{\mathfrak{p}_1,0,\mu_1} /\RVerma_{\mathfrak{p}_1,-\rho,\mu_1}\xrightarrow{\cong} M^{\vee}_{\overline{l} = 0}\;.
\end{align*}
By Proposition \autoref{Prop:Mlgeqmu2+2eqmu2eqmu2+1simple} 
and Proposition \autoref{Prop:structureofMbgeqmu2etc}(3)
 $M_{\overline{l} = 0}$ is simple and self-dual and hence the unique
simple quotient $\Lsimple_{\mathfrak{p}_1, 0, \mu_1}$ of $\RVerma_{\mathfrak{p}_1, 0, \mu_1}$ is $\Lsimple_{\mathfrak{p}_1, 0, \mu_1} = \RVerma_{\mathfrak{p}_1, 0, \mu_1}/\RVerma_{\mathfrak{p}_1, -\rho, \mu_1}$. The self-duality of $M_{\overline{l}=0}$
implies the one of $\Lsimple_{\mathfrak{p}_1, 0, \mu_1}$. 
 \\
(5) The argument is completely analogous to the one of (4). 
\end{proof}

 \subsection{Structure of $M_{\overline{l} \leq -1}$} \label{Rem:characterofbleqmu2}
 
 \begin{Prop} \label{Prop:Mleqmu2-1} \noindent
\begin{enumerate}
\item $M_{\overline{l} \leq -1}$ is a simple $\mathfrak{sl}_3$-module.
\item $M_{\overline{l} \leq -1} \cong M_{\overline{l} \leq -1}^{\vee}$
\end{enumerate}
\end{Prop}

\begin{proof}
(1) By Remark \autoref{Rem:simultaneousevecslieinsubmodule}
it suffices to show that $w_{k_0, l_0, m_0}$ generates $M_{\overline{l} \leq -1}$
for given $k_0 \in \Z, \overline{l_0} \leq -1, m_0 \geq 0$.
We look at \eqref{eq:actionofbasisofsl3onwklm}. Acting with $e_{12}$ we find $w_{k_0, l_0, 0}\in S$. Acting with $e_1$ we find $w_{k, l_0, 0}\in S$ for $k \leq k_0$. 
Acting then with $e_2$ we find $w_{k,l,0} \in S$ for $k \leq k_0, l \leq l_0$.
Acting then with $f_{12}$ we find $w_{k,l,m} \in S$ for $k \leq k_0, l \leq l_0, m \geq 0$.
Acting then with $f_1$ we find $w_{k,l,m}\in S$ for $k \in \Z, l \leq l_0, m \geq 0$. 
Acting finally with $f_2$ we find $w_{k,l,m}\in S$ for $k \in \Z, \overline{l} \leq -1, m \geq 0$. 
\\

(2) Let $x_{0, \mu_2-1, 0}\in \C^{\times}$. One checks
that $\phi: M^{\vee}_{\overline{l} \leq -1} \rightarrow
M_{\overline{l} \leq -1}$ is an isomorphism of $\mathfrak{sl}_3$-modules
if we set $\phi(\eta_{k,l,m})=x_{k,l,m} w_{k,l,m}$ with
\begin{align*}
x_{k,l,m} = \frac{(-1)^m m!}{(\overline{k}+\overline{l})^{(m)}}\; \; \left(-\frac{\overline{k}}{\mu_1}\right) x_{0, \mu_2-1,0} \begin{cases}\frac{(\overline{l}-\mu_1-1)^{(k)}}{(-\mu_1-1^{(k)}} & k \geq 0 \\ \frac{(-\mu_1-2)_{(-k)}}{(\overline{l}-\mu_1-2)_{(-k)}} & k < 0 \end{cases} \times \begin{cases} 
1 & \overline{l} = -1 \\-\frac{\overline{l}(3 \cdots (-\overline{l}+1))}{(\mu_1+3)^{(-\overline{l}-1)}}
 & \overline{l} \leq -2\end{cases}\;.
\end{align*} 
\end{proof}

\section*{Acknowledgements}
We are grateful to B. Feigin, G. Felder, and M. Finkelberg for discussions related to this article. 

\bibliographystyle{alpha}
\bibliography{references}
\end{document}